\documentclass[
superscriptaddress, 
amsmath,
amssymb,
aps,
rmp,
floatfix,
]{revtex4-2}
\usepackage{geometry}
\usepackage{graphicx}
\usepackage{dcolumn}
\usepackage{bm}
\usepackage{hyperref}
\usepackage[mathlines]{lineno}

\usepackage{stmaryrd} 
\usepackage{xcolor, soul}
\usepackage{placeins} 
\usepackage{subcaption} 
\usepackage{enumerate} 
\usepackage{yhmath} 
\usepackage{ragged2e}
\usepackage{enumitem}
\usepackage{amssymb}
\usepackage{amsthm}
\usepackage{stmaryrd}
\usepackage{diagbox}
\usepackage{array}
\usepackage{calrsfs}
\usepackage{colortbl}
\usepackage{makecell}
\usepackage{multirow}
\usepackage{geometry}
\usepackage{physics}
\usepackage{lipsum}
\usepackage{appendix} 
\usepackage{todonotes}
\usepackage{caption}

\newtheorem{lemma}{Lemma}
\newtheorem{proposition}{Proposition}
\newtheorem{theorem}{Theorem}

\newtheorem{definition}{Definition}

\newcommand{\g}{\mathfrak{g}}

\newcommand{\E}{\mathbb{E}}

\begin{document}

\preprint{APS/123-QED}

\title{Another Mar\v{c}enko-Pastur law for Kendall’s tau}

\author{Pierre Bousseyroux}
\email{pierre.bousseyroux@polytechnique.edu}
\affiliation{Chair of Econophysics and Complex Systems, Ecole Polytechnique, 91128 Palaiseau Cedex, France}
\affiliation{LadHyX UMR CNRS 7646, \'Ecole Polytechnique, 91128 Palaiseau Cedex, France}

\author{Tomas Espana}
\email{tomas.espana@student-cs.fr}
\affiliation{Chair of Econophysics and Complex Systems, Ecole Polytechnique, 91128 Palaiseau Cedex, France}
\affiliation{LadHyX UMR CNRS 7646, \'Ecole Polytechnique, 91128 Palaiseau Cedex, France}

\author{Matteo Smerlak}
\email{matteo.smerlak@cfm.com}
\affiliation{Capital Fund Management, 23 rue de l'Universit\'e, 75007 Paris, France}

\date{\today}

\begin{abstract}
    Bandeira \textit{et al.} (2017) show that the eigenvalues of the Kendall correlation matrix of $n$ i.i.d. random vectors in $\mathbb{R}^p$ are asymptotically distributed like $1/3 + (2/3)Y_q$, where $Y_q$ has a Marčenko-Pastur law with parameter $q=\lim(p/n)$ if $p, n\to\infty$ proportionately to one another. Here we show that another Marčenko-Pastur law emerges in the "ultra-high dimensional" scaling limit where $p\sim q'\, n^2/2$ for some $q'>0$: in this quadratic scaling regime, Kendall correlation eigenvalues converge weakly almost surely to $(1/3)Y_{q'}$.
\end{abstract}

\keywords{} 

\maketitle



\FloatBarrier
\section{Introduction}

Since the seminal work of \citep{marchenko1967distribution}, it is known that the eigenvalues of the sample correlation matrix for i.i.d. data $X\in \mathbb{R}^{p\times n}$ follow the Marčenko-Pastur (MP) law $Y_q$ when the number of features $p$ scales linearly with the number of observations $n$, i.e. $p\sim qn$ for some $q>0$ as $p, n\to \infty$. This result provides a useful null hypothesis when testing for dependence in high-dimensional data, with applications ranging from genomics to finance. 

However, from both a practical and a conceptual standpoint, it is often preferable to use measures of dependence other than sample (or 'Pearson') correlation coefficients. Indeed, as is well known, Pearson correlations are neither invariant with respect to increasing transformations of the data nor robust to outliers or heavy-tailed data. A popular alternative is Kendall's tau, which measures rank correlations and shares neither of these deficiencies. Because the columns of the Kendall correlation matrix $\boldsymbol{\tau}$ are not independent, however, the classical Marčenko-Pastur theorem does not apply and new tools are required. 

Using an elegant (Hoeffding) decomposition, \citep{bandeira2017marvcenkopastur} proved that the spectrum of the Kendall matrix follows an affine transformation of the Marčenko-Pastur law: under the usual assumptions on the data matrix, they show that Kendall eigenvalues are distributed like $(1/3) + (2/3)Y_q$. Subsequent papers established the Tracy-Widom law of its largest eigenvalues \citep{tracy_widom_kendall}, the central limit theorem of linear spectral statistics \citep{clt_lss}, and the limiting spectral distribution (LSD) of $\boldsymbol{\tau}$ under dependence conditions \citep{li2023eigenvalues}. 

One property of Kendall's correlation matrix which contrasts with the Wishart ensemble is that $\boldsymbol{\tau}$ does not develop zero eigenvalues when $p>n$; in fact the first zero eigenvalue only emerges when $p=\binom{n}{2} = n(n-1)/2$. This motivates us to consider the behavior of Kendall eigenvalues in the `ultra-high dimensional" regime where $p$ scales quadratically, rather than linearly, with $n$. Moreover, in many practical high-dimensional problems, the number of samples remains limited while the feature dimension increases much more rapidly, making quadratic growth a realistic alternative to the usual linear scaling---for example in quantitative finance when constructing large portfolios (see, e.g., \citep{kendall_markowitz}). Extending the results of \citep{bandeira2017marvcenkopastur}, we show in Theorem \ref{theo:kendall_quad} that the LSD of Kendall matrices converges (weakly almost surely) to $(1/3)Y_{q'}$ with $q' = \lim 2p/n^2\in(0,\infty)$. The next paragraph introduces several notations inspired by \citep{tracy_widom_kendall}. 

\section{Main result}
\paragraph*{Kendall Correlation Matrix.}

Let $\mathbf{x} = (x_1, \ldots, x_p)^\top \in \mathbb{R}^p$. We assume that all the components of $\mathbf{x}$ are independent random variables, with absolutely continuous densities with respect to the Lebesgue measure. We do not require the components to be identically distributed, and no moment assumption on the components of $\mathbf{x}$ is needed. Let $\mathbf{x}_j = (x_{1j}, \ldots, x_{pj})^\top$ be $n$ i.i.d.\ samples of $\mathbf{x}$. Hereafter we write, for any integer $k \geq 1$, $[k] = \{1, 2, \ldots, k\}$. We also denote by $\mathbf{X} = (x_{ij})_{p,n}$ the $p \times n$ data matrix. 

From the data matrix $\mathbf{X}$ we can build the \textit{Kendall correlation matrix}, based on Kendall's tau correlation coefficient. The matrix is defined as follows:

\begin{definition}[Kendall Correlation Matrix]\label{def:kendall}
  Consider a data matrix $\mathbf{X}$ as defined previously. For any $k \in [p]$, we denote
  \begin{align}\label{defv}
    v_{k, (ij)} := \mathrm{sign}(x_{ki} - x_{kj}), \quad \text{for any} \quad i < j, \quad (i,j) \in [n]^2
  \end{align}
  and let
  \begin{align}
    \boldsymbol{\theta}_{(ij)} := \frac{1}{\sqrt{M}} \left( v_{1,(ij)}, \ldots, v_{p,(ij)} \right)^\top
  \end{align}
  where
  \begin{align}
    M = M(n) := \frac{n(n-1)}{2}.
  \end{align}
  Kendall's correlation matrix is defined as the following sum of rank-one projectors
  \begin{align}
    \boldsymbol{\tau} = \boldsymbol{\tau}_n := \sum_{i<j} \boldsymbol{\theta}_{(ij)} \boldsymbol{\theta}_{(ij)}^\top = \boldsymbol{\Theta} \boldsymbol{\Theta}^\top
  \end{align}
  where $\boldsymbol{\Theta}$ (implicitly $n$ dependent) is the $p \times \binom{n}{2}$ matrix defined by its columns as follows
  \begin{align}
    \boldsymbol{\Theta} := \left( \boldsymbol{\theta}_{(12)}, \ldots, \boldsymbol{\theta}_{(1n)}, \boldsymbol{\theta}_{(23)}, \ldots, \boldsymbol{\theta}_{(2n)}, \ldots, \boldsymbol{\theta}_{(n\!-\!1 \, n)} \right).
  \end{align}
\end{definition}

Observe that the rank-one projectors $\boldsymbol{\theta}_{(ij)} \boldsymbol{\theta}_{(ij)}^\top$ are not independent. For instance, $\boldsymbol{\theta}_{(ij)} \boldsymbol{\theta}_{(ij)}^\top$ and $\boldsymbol{\theta}_{(ik)} \boldsymbol{\theta}_{(ik)}^\top$ are correlated even if $j \neq k$. Moreover, the entries of the $p \times p$ matrix $\boldsymbol{\tau}$ are
  \begin{align}\label{eq:kendall_def_uni}
    \boldsymbol{\tau}_{(kl)} = \frac{2}{n(n-1)} \sum_{i<j} \mathrm{sign}(x_{ki} - x_{kj}) \mathrm{sign}(x_{li} - x_{lj})
  \end{align}
  which correspond to the Kendall correlation coefficient between the samples of $x_k$ and $x_l$.

Let $\lambda_1^{(n)}, \ldots, \lambda_p^{(n)}$ be the $p$ eigenvalues of $\boldsymbol{\tau}_n$. We note the empirical spectral distribution (ESD) of $\boldsymbol{\tau}_n$ by
\begin{align}
  F^{\boldsymbol{\tau}_n}(x) := \frac{1}{p} \sum_{i=1}^p I(\lambda_i^{(n)} \leq x)
\end{align}
where $I(\cdot)$ is the indicator function. When the sequence $(F^{\boldsymbol{\tau}_n})_n$ of ESDs associated with the random matrices $(\boldsymbol{\tau}_n)$ converges to a deterministic density function $F^{\boldsymbol{\tau}}$ almost surely, this function is called the limiting spectral distribution (LSD) of the sequence. Such convergence is referred to as almost sure weak convergence.

\paragraph*{Kendall's Limiting Spectral Distribution}
In the \textit{linear} scaling $p = \mathcal{O}(n)$, the spectrum of $\boldsymbol{\tau}$ follows an affine transformation of the MP distribution:

\begin{theorem}[Theorem 1 of \citep{bandeira2017marvcenkopastur}]\label{theo:bandeira}
  Under the assumptions of Def.\ \ref{def:kendall}, and if $q_n = p/n \rightarrow q >0$ when $n \rightarrow \infty$, the sequence $(F^{\boldsymbol{\tau}_n})_n$ converges weakly in probability to the distribution associated with the random variable $(1/3) + (2/3) Y_{q}$, where $Y_{q}$ has a MP law with parameter $q > 0$.
\end{theorem}

From Eq.\ (\ref{eq:kendall_def_uni}), we observe that the entries of the Kendall correlation matrix $\boldsymbol{\tau}$ are in fact scalar products in dimension $\binom{n}{2} = \mathcal{O}(n^2)$. Henceforth we therefore assume the quadratic scaling
\begin{align}\label{eq:assumption}
  p = p(n), \qquad q'_n := \frac{2p}{n(n-1)} \rightarrow q' > 0, \qquad \mathrm{if} ~ n \rightarrow \infty,
\end{align}
for some positive constant $q'$.

\begin{theorem}\label{theo:kendall_quad}
  Consider a data matrix $\mathbf{X}$ as defined previously. Suppose that assumption (\ref{eq:assumption}) holds. Then, the empirical spectral distribution of $\boldsymbol{\tau}$ converges weakly almost surely to the distribution $(1/3) Y_{q'}$ where $Y_{q'}$ has a MP law with parameter $q' > 0$.
\end{theorem}
Intuitively, Theorem \ref{theo:bandeira} tells us that the matrix $\boldsymbol{\tau}$ behaves like $(2/3)\mathbf{W}(q) + (1/3)\mathbf{I}_p$, where $\mathbf{W}(q)$ is a $p\times p$ Wishart matrix of parameter $q$, when $n$ and $p$ are sufficiently large while $q = p/n$ remains of order $1$. In the quadratic regime, we have far fewer data points: $n$ is of order $\sqrt{p}$, and thus $q$ becomes very large. According to Def. \ref{definition:mp_law}, the matrix $\mathbf{W}(q)$ then consists mostly of eigenvalues equal to $0$, but approximately $1/q$ of the eigenvalues remain in the bulk, located between $\lambda_{-}$ and $\lambda_{+}$, and are of order $q$. The convergence in law does not capture these eigenvalues that escape to infinity, and one may then think of $\mathbf{W}(q)$ as behaving like the zero matrix when $q \to +\infty$. Our main theorem explains that the term $(1/3)\mathbf{I}_p$ will spread out in the quadratic regime and itself assume a MP distribution. We can informally state that $\boldsymbol{\tau} \approx (1/3)\mathbf{W}\left(\frac{2p}{n(n-1)}\right) + (2/3)\mathbf{W}\left(p/n\right)$ when $n$ and $p$ are large. In the linear regime, i.e. when $n$ and $p$ are of the same order, the term $\mathbf{W}\left(\frac{2p}{n(n-1)}\right)$ behaves like the identity matrix, and we recover Theorem \ref{theo:bandeira}. Conversely, in the quadratic regime, $\mathbf{W}\left(p/n\right)$ disappears as explained above, and we are left only with $(1/3)\mathbf{W}\left(\frac{2p}{n(n-1)}\right)$.

However, we must be cautious with such reasoning. We emphasize that it is not rigorous because, for instance, claiming that $\boldsymbol{\tau} \approx (1/3)\mathbf{W}\left(\frac{2p}{n(n-1)}\right)$ in the quadratic regime leads to issues, such as with the trace. On the left-hand side, we have an object whose first moment is $1$, while on the right-hand side, it is $(1/3)$. The issue arises from the very large eigenvalues of $(2/3)\mathbf{W}\left(p/n\right)$ that escape to infinity. To overcome this, we will first rely on the Hoeffding decomposition and then apply random matrix theory methods to the leading-order term.

This will be the focus of the next section. The two most well-known approaches to prove the Marčenko-Pastur theorem (as presented in \citep{bai2010spectral}) are: one based on the moment method and the other using the cavity method, whose general principles are detailed in the appendix (see Appendix \ref{RMT}). To prove the main theorem, it suffices to adapt these two classical proofs. In this article, we have chosen to focus solely on the cavity method, with its key ideas revisited in the next section. Section IV will then present the numerical results.

\FloatBarrier

\section{Proof of Theorem \ref{theo:kendall_quad}}

A key point of the proof strategy used by \citep{bandeira2017marvcenkopastur} is the following Hoeffding decomposition \citep{hoeffding}
\begin{align}\label{eq:hoeffding}
  v_{k,(ij)} = u_{k,(i\cdot)} + u_{k,(\cdot j)} + \bar v_{k,(ij)}
\end{align}
where
\begin{align}
   & u_{k,(i\cdot)} = \mathbb{E}[\mathrm{sign}(x_{ki} - x_{kj}) | x_{ki}],
   & u_{k,(\cdot j)} = \mathbb{E}[\mathrm{sign}(x_{ki} - x_{kj}) | x_{kj}]
\end{align}
and Eq.\ (\ref{eq:hoeffding}) defines $\bar v_{k,(ij)}$. Note that the three terms on the right hand-side of Eq.\ (\ref{eq:hoeffding}) are centered and uncorrelated. Since the sign function is odd, we can rewrite
\begin{align}\label{eq:hoeffding_sign}
  v_{k,(ij)} = u_{k,(i\cdot)} - u_{k,(j \cdot)} + \bar v_{k,(ij)}
\end{align}
By introducing
\begin{align}
  \boldsymbol{\bar{\Theta}}_{(ij)} := \frac{1}{\sqrt{M}} \left( \bar v_{1,(ij)}, \ldots, \bar v_{p,(ij)} \right)^\top
\end{align}
we can extend Eq.\ (\ref{eq:hoeffding}) to the multivariate case and write
\begin{align}\label{eq:decomp_linear}
  \boldsymbol{\theta}_{(ij)} \boldsymbol{\theta}_{(ij)}^\top = \boldsymbol{\bar{\Theta}}_{(ij)} \boldsymbol{\bar{\Theta}}_{(ij)} ^\top + \mathbf{A}_{(ij)}
\end{align}
which defines the term $\mathbf{A}_{(ij)}$. In the linear scaling $p = \mathcal{O}(n)$, \citep{bandeira2017marvcenkopastur} show that the term $\mathbf{A} = \sum_{i<j} \mathbf{A}_{(ij)}$ accounts for the MP law $(2/3) Y_{q}$ of Theorem \ref{theo:bandeira}, while the term
\begin{align}\label{eq:defH}
  \mathbf{H} = \mathbf{H}_n := \sum_{i<j} \boldsymbol{\bar{\Theta}}_{(ij)} \boldsymbol{\bar{\Theta}}_{(ij)} ^\top = \boldsymbol{\bar \Theta} \boldsymbol{\bar \Theta}^\top
\end{align}
accounts for the intercept $(1/3)$. 

To prove Theorem \ref{theo:kendall_quad} in the quadratic regime, the first step is to prove that $\boldsymbol{\tau}$ and $\mathbf{H}$ have the same LSD. This is the goal of the following proposition proved in Appendix \ref{pro}.

\begin{proposition}\label{prop:prop1}
    $\|F^{\boldsymbol{\tau}} - F^{\mathbf{H}} \|_{\infty} \underset{n \to \infty}{\longrightarrow} 0$.
\end{proposition}

The second part of the proof involves studying the LSD of \(\mathbf{H}\). One can adapt the method of moments used in the proof of the Marčenko-Pastur law and note that dependent pairs can be ignored. Here, we present the cavity method for brevity. We encourage readers to refer to the recap of the proof of the Marčenko-Pastur theorem using this method in Appendix~\ref{RMT}. We recall that there are two steps. The first step consists of showing that
\begin{equation}\label{step1}
    \g^p(z) - \E(\g^p(z)) \underset{p \to +\infty}{\longrightarrow} 0
\end{equation}
almost surely, where \(\g^p(z)\) denotes the empirical Stieltjes transform of a realization of the matrix \(\mathbf{H}\) of size \(p \times p\), which we will prove in Appendix~\ref{pro3}.

The second step consists of showing that \(\E(\g^p(z))\) converges to \(\g_{\rho}(z)\), the Stieltjes transform associated with \(\rho\), the distribution of \(Y_{q'}/3\). It suffices to adapt the outline of the proof presented in Appendix~\ref{RMT} by considering the matrix $\mathbf{H}$, defined in Eq.\ (\ref{eq:defH}), instead of a Wishart matrix. If we establish the following proposition, the corresponding proof of \citep{bai2010spectral} for Wishart matrices proceeds in the same way, as \(\Tilde{\mathbf{H}}\) corresponds to a Kendall matrix of size \((p\!-\!1) \times (p\!-\!1)\) associated with a sample of size \( n \).

\begin{proposition}
    \begin{equation}\label{key}
C := \mathbb{E}\left[\sum_{2 \leq a,b \leq p} \mathbf{H}_{1a}\,\left[\mathbf{G}_{\tilde{\mathbf{H}}}(z)\right]_{ab}\,\mathbf{H}_{b1}\right]
= \frac{-q'' + q''\,z\, \mathbb{E}[\g^{p-1}(z)]}{3}
\end{equation}where $q'' = \frac{2(p-1)}{n(n-1)}$, $\tilde{\mathbf{H}}$ denotes the matrix $\mathbf{H}$ with its first row and first column removed, $\mathbb{E}[\g^{p-1}(z)]$ is the expected value of the Stieltjes transform of the random empirical spectral distribution corresponding to a realization of the $(p-1) \times (p-1)$ matrix $\tilde{\mathbf{H}}$, and $\mathbf{G}_{\tilde{\mathbf{H}}}(z)$ is the resolvent matrix associated with $\tilde{\mathbf{H}}$. 
\end{proposition}
Let us now prove Eq.\ \eqref{key}. From the definition of $\mathbf{H}$, we have:

\begin{multline}
    C = \left(\frac{2}{n(n-1)}\right)^2 \sum_{2 \leq a, b \leq p} \sum_{1 \leq i_1 < j_1 \leq n} \sum_{1 \leq i_2 < j_2 \leq n} \mathbb{E}\left[\bar{v}_{1,(i_1j_1)} \bar{v}_{a,(i_1j_1)} \bar{v}_{1,(i_2j_2)} \bar{v}_{b,(i_2j_2)} [\mathbf{G}_{\Tilde{\mathbf{H}}}(z)]_{ab}\right].
\end{multline}
We now observe that $[\mathbf{G}_{\Tilde{\mathbf{H}}}(z)]_{ab}$ only depends on $\Tilde{\mathbf{X}}$, which is the matrix $\mathbf{X}$ with its first row removed. We can therefore write, by independence:

\begin{multline}
    \mathbb{E}\left[\bar{v}_{1,(i_1j_1)} \bar{v}_{a,(i_1j_1)} \bar{v}_{1,(i_2j_2)} \bar{v}_{b,(i_2j_2)} [\mathbf{G}_{\Tilde{\mathbf{H}}}(z)]_{ab}\right] 
    = \mathbb{E}\left[\mathbb{E}\left[\bar{v}_{1,(i_1j_1)} \bar{v}_{1,(i_2j_2)}\right] \bar{v}_{a,(i_1j_1)} \bar{v}_{b,(i_2j_2)} [\mathbf{G}_{\Tilde{\mathbf{H}}}(z)]_{ab}\right].
\end{multline}
One can check that
\begin{equation}
    \mathbb{E}\left[\bar{v}_{1,(i_1j_1)} \bar{v}_{1,(i_2j_2)}\right] =
    \begin{cases} 
        0 & \text{if } (i_1, j_1) \neq (i_2, j_2), \\
        (1/3) & \text{otherwise},
    \end{cases}
\end{equation}and so express \( C \) as
\begin{equation}
    C = (1/3) \left(\frac{2}{n(n-1)}\right)^2 \sum_{2 \leq a, b \leq p} \sum_{1 \leq i < j \leq n} \mathbb{E}\left[\bar{v}_{a,(ij)} \bar{v}_{b,(ij)} [\mathbf{G}_{\Tilde{\mathbf{H}}}(z)]_{ab}\right],
\end{equation}
which simplifies to
\begin{equation}
    (1/3)\frac{2}{n(n-1)} \mathbb{E} \left[ \tr\left(\Tilde{\mathbf{H}} \mathbf{G}_{\Tilde{\mathbf{H}}}(z)\right) \right].
\end{equation}
Finally, this can be rewritten as
\begin{equation}
    \frac{-q'' + q''z \mathbb{E}\left[\g^{p-1}(z)\right]}{3},
\end{equation}
where \( q'' = \frac{2(p-1)}{n(n-1)} \), and \(\Tilde{\mathbf{H}}\) corresponds to a Kendall matrix of size \((p-1) \times (p-1)\) associated with a sample of size \( n \). We have thus proven Equation \eqref{key}, which completes the proof of the result.

Furthermore, in the case where the components of the random vector $\mathbf{x}\in \mathbb{R}^p$ are i.i.d., the attentive reader will have noticed that the entire proof holds when replacing the definition of $v_{k, (ij)}$ with
\begin{equation}
    v_{k, (ij)} = \phi(x_{ki}, x_{kj}),
\end{equation}
where $\phi: \mathbb{R}^2 \to \mathbb{R}$ is a bounded and antisymmetric function. One would then obtain that, in the same quadratic regime, the limiting spectral distribution of the new matrix $\boldsymbol{\tau}$ converges almost surely to the distribution $\alpha Y_{q'}$, where $Y_{q'}$ is a random variable distributed according to the MP law with parameter $q'$, and

\begin{equation}\label{eq:gcc}
    \alpha := \mathbb{E}\left[(\phi(\mathbf{x}_1, \mathbf{x}_2) - \mathbb{E}[\phi(\mathbf{x}_1, \mathbf{x}_2) | \mathbf{x}_1] - \mathbb{E}[\phi(\mathbf{x}_1, \mathbf{x}_2) | \mathbf{x}_2])^2\right]
\end{equation}where $\mathbf{x}_1$ and $\mathbf{x}_2$ denote the first two random coordinates of the vector $\mathbf{x}$. A classical example of antisymmetric function is $\phi(x,y)\!=\!\psi(y-x)$ for $\psi:\mathbb{R}\to\mathbb{R}$ odd \citep{daniels}. When $\psi(x) = \text{sgn}(x)$, we recover Kendall's correlation coefficient, whereas $\psi(x)\!=\!x$ corresponds to Pearson's correlation coefficient (up to an appropriate normalization). In the former case, the correlation coefficient is independent of the marginal distributions of the random variables, while in the latter it is highly sensitive to them. This motivates the exploration of \textit{intermediate} kernels of the form $\psi(x)\!=\!\tanh(x)$, which aim to strike a balance between these two regimes. However, the parameter $\alpha$ in Eq.~(\ref{eq:gcc}) becomes data-dependent and is no longer universal, as it is for Kendall and Pearson coefficients \citep{benaych2025spectral}.

\section{Numerics}

In this section, we plot and discuss the histogram of the eigenvalues of Kendall's matrix to validate our theoretical results. Fig.\ \ref{fig:hist_H} shows the histogram of the empirical eigenvalues of $\mathbf{H}$ against the theoretical formula derived in Theorem \ref{theo:kendall_quad}. The empirical eigenvalues of the Kendall matrix exhibit a more intriguing phenomenon. Indeed, we see that both the eigenvalues of the linear MP $(2/3)\mathbf{W}\left(p/n\right)$ and of the quadratic MP $(1/3)\mathbf{W}\left(\frac{2p}{n(n-1)}\right)$ appear. Fig.\ \ref{fig:eigval_kendall} shows how the eigenvalues of the linear MP increase and how the bulk of eigenvalues of the quadratic MP spread around their average value $(1/3)$, as $p$ increases. Heuristically, the linear eigenvalues do not affect the Kendall spectrum since their distribution converges to a Dirac mass at zero and an infinite eigenvalue. This infinite eigenvalue scales at $(1+\sqrt{q})^2 \sim q$ with probability $\frac{1}{q}$ explaining why the first moment of the Kendall matrix exists (equal to 1) but not its other moments. The Python code used for simulations is available at \url{https://github.com/espanato/kendall_eigenvalues}.     

\begin{figure}[htbp]
\centering
\includegraphics[width=0.5\linewidth]{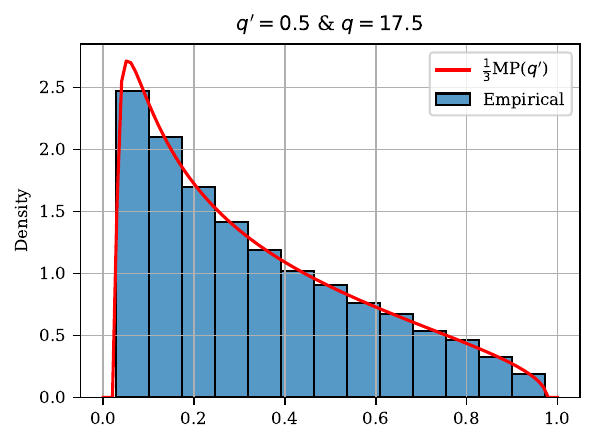}
\caption{\justifying Histogram of the eigenvalues of $\mathbf{H}$ for $n\!=\!70$ and $p\!=\!1225$ (${q'}\!\approx\!0.5$ and $q\!\approx 17.5\!$). The superimposed red line is the density function of $(1/3)Y_{q'}$.}
\label{fig:hist_H}
\end{figure}

\begin{figure}[ht]
  \centering
  \includegraphics[width=0.6\linewidth]{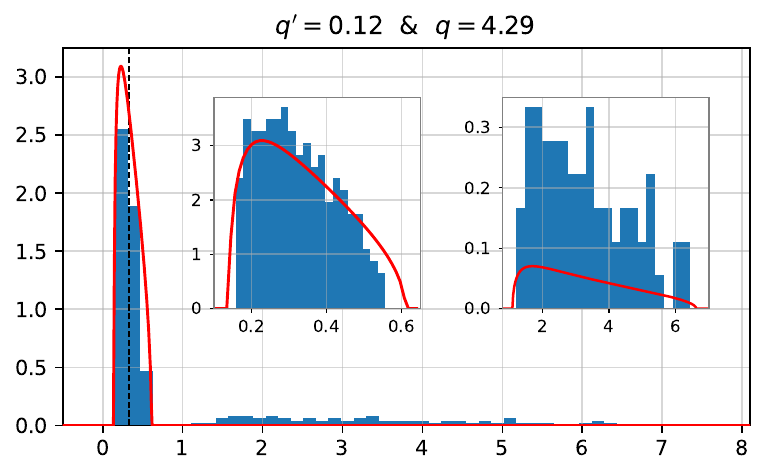}
  \vspace{0.8em}
  \includegraphics[width=0.6\linewidth]{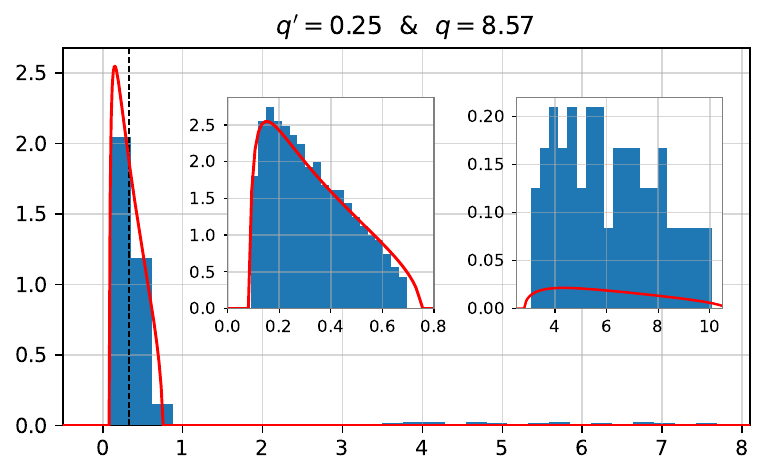}
  \caption{\justifying Histogram of the eigenvalues of $\boldsymbol{\tau}$ for $n=70$ (top: $p=300$, bottom: $p=600$). 
  The superimposed red line is the density of $(1/3)Y_{q'}$. The vertical black dashed line marks $(1/3)$. The two inset plots show zoomed-in views of the eigenvalue distribution around selected regions. The left  inset plot shows the 'quadratic' eigenvalues and superimposes the density of $(1/3)Y_{q'}$. The right inset plot shows the 'linear' eigenvalues and superimposes the density of $(2/3)Y_{q}\!+\!1/3$.}
  \label{fig:eigval_kendall}
\end{figure}

\FloatBarrier
\begin{acknowledgments}
\vspace{-0.3cm}
We thank Florent Benaych-Georges, Jean-Philippe Bouchaud, Victor Le Coz and Afonso Bandeira, who contributed to our research through fruitful discussions.
This research was conducted within the Econophysics \& Complex Systems Research Chair, under the aegis of the Fondation du Risque, the Fondation de l'Ecole Polytechnique, the Ecole Polytechnique and Capital Fund Management.
\end{acknowledgments}

\FloatBarrier
\newpage
\appendix
\section{RMT Reminders}\label{RMT}

\begin{definition}\label{definition:mp_law}
  The Marčenko-Pastur law $\rho$ of parameter $q > 0$ is the probability distribution
  \begin{equation*}
    \rho(\lambda) = \frac{1}{2 \pi q \lambda} \sqrt{[(\lambda_+ - \lambda)(\lambda - \lambda_-)]_{+}} + (1 - \frac{1}{q})\delta(\lambda) \Theta(q-1)
  \end{equation*}
  where we denote $\lambda_{+} = (1 + \sqrt{q})^2, \, \lambda_{-} = (1 - \sqrt{q})^2, \, [a]_{+} := \mathrm{max}(0, a)$ for any $a \in \mathbb{R}$ and \newline
  $\Theta(q-1) := \begin{cases}
      0 \,\,  \text{if} \,\,\, q \leq 1 \\
      1 \,\, \text{if} \,\,\, q > 1
    \end{cases}$.
\end{definition}

A convenient way to study the eigenvalues of a random matrix $\mathbf{W}$ of size $p \times p$ is to look at the resolvent matrix

\begin{equation}
    \mathbf{G}_{\mathbf{W}}(z) = (z - \mathbf{W})^{-1},
\end{equation}
where $z$ is a complex number outside the spectrum of $\mathbf{W}$, especially its normalized trace 

\begin{equation}
    \g_{\mathbf{W}}^{p}(z) := \frac{1}{p} \tr(\mathbf{G}_{\mathbf{W}}(z)).
\end{equation}
This can be written as 

\begin{equation}
    \g_{\mathbf{W}}^{p}(z) = \int \frac{\rho^p (\lambda)}{z - \lambda} \, d\lambda,
\end{equation}
where $\rho^p$ is the empirical spectral distribution of $\mathbf{W}$.
The Stieltjes transform of a given distribution $\rho$ is the function $\g_{\rho}$ defined as

\begin{equation}
    \g_{\rho}(z) = \int \frac{\rho(\lambda)}{z - \lambda} \, d\lambda
\end{equation}where $z$ is a complex number outside the support of $\rho$. 

One way to prove the almost sure weak convergence of the empirical spectral distribution of $\mathbf{W}$ towards the Marčenko-Pastur distribution with parameter $q$ is to show that

\begin{equation}\label{step2}
    \mathbb{E}[\g_{\mathbf{W}}^{p}(z)] \underset{p \to +\infty}{\longrightarrow} \g_{\rho}(z),
\end{equation}
where $\rho$ is the Marčenko-Pastur distribution with parameter $q$. 

However, this result is not sufficient to guarantee the almost sure weak convergence of the empirical spectral distribution. We must also demonstrate that  
\begin{equation}\label{step1}
    \g_{\mathbf{W}}^{p}(z) \underset{p \to +\infty}{\longrightarrow} \mathbb{E}[\g_{\mathbf{W}}^{p}(z)]
\end{equation}
almost surely.

This approach is known as the method of the Stieltjes transform. A rigorous proof of the limiting spectral distribution of a $p \times p$ Wishart matrix $\mathbf{W}$ with parameter $p/n$ using this method can be found in \citep{bai2010spectral}, which refers to Eq.~\eqref{step1} as Step 1 and Eq.~\eqref{step2} as Step 2. A key part of Step 2 is to prove that:

\begin{align}
    \mathbb{E}\left[\sum_{2\leq a, b \leq p} [\mathbf{W}]_{1a} [\mathbf{G}_{\Tilde{\mathbf{W}}}(z)]_{ab} [\mathbf{W}]_{b1}\right] = -\frac{p-1}{n} + \frac{p-1}{n}z \mathbb{E}[\g^{p-1}(z)],
\end{align}
where $\mathbf{G}_{\Tilde{\mathbf{W}}}(z)$ is the resolvent matrix of size $(p-1) \times (p-1)$ associated with $\Tilde{\mathbf{W}}$, which is the matrix $\mathbf{W}$ with its first row and first column removed. Here, $\g^{p-1}(z)$ is the random Stieltjes transform of the empirical distribution of $\Tilde{\mathbf{W}}$, which is still a $(p-1)\times (p-1)$ Wishart matrix with parameter $(p-1)/n$. The rest of the proof involves using the cavity method.

\section{Proof of Proposition \ref{prop:prop1}}\label{pro}

\noindent\textbf{Proposition \ref{prop:prop1}.} \itshape
$\|F^{\boldsymbol{\tau}} - F^{\mathbf{H}}\|_{\infty} \underset{n \to \infty}{\longrightarrow} 0$.
\normalfont

\begin{proof}
Using capital letters to extend Eq.\ (\ref{eq:hoeffding_sign}) to $\mathbb{R}^p$ vectors gives
\begin{align}
    V_{(ij)} = U_i - U_j + \bar V_{(ij)}
\end{align}
From Eq.\ (\ref{eq:decomp_linear}), we can readily show that

\begin{align}
    \mathbf{A}_{(ij)} = \mathbf{A}_{(ij)}^{(1)} + \mathbf{A}_{(ij)}^{(2)} + \left( \mathbf{A}_{(ij)}^{(2)} \right)^\top 
\end{align}
where
\begin{align}
    & \mathbf{A}_{(ij)}^{(1)} := \frac{1}{M} \left(U_i - U_j\right) \left(U_i - U_j \right)^\top \\
    & \mathbf{A}_{(ij)}^{(2)} := \frac{1}{M} \bar V_{(ij)} \left(U_i - U_j \right)^\top
\end{align}
Considering the sample mean $\langle U \rangle = \frac{1}{n} \sum_{i=1}^n U_i$, we notice that
\begin{align}
    \sum_{i<j} \mathbf{A}_{(ij)}^{(1)} = \frac{2}{n-1} \sum_{i=1}^n \left(U_i - \langle U \rangle \right)\left(U_i - \langle U \rangle \right)^\top
\end{align}
which is a sum of $n$ rank-one projectors. Similarly, a straightforward rearrangement of terms shows that $\sum_{i<j} \mathbf{A}_{(ij)}^{(2)} + \left( \mathbf{A}_{(ij)}^{(2)} \right)^\top$ can be written as the sum of $2n$ rank-two Hermitian matrices. Thus, in the quadratic regime, the application of the following Lemma \ref{lemma:rank_ineq} yields:

\begin{align}
    \|F^{\boldsymbol{\tau}} - F^{\mathbf{H}} \|_{\infty} & \leq \frac{1}{p} \mathrm{rank} \left( \sum_{i<j}  \mathbf{A}_{(ij)}^{(1)} + \mathbf{A}_{(ij)}^{(2)} + \left(\mathbf{A}_{(ij)}^{(2)} \right)^\top \right)\\
    & \leq \frac{5n}{p} \rightarrow 0
\end{align}
\end{proof}

\begin{lemma}[Theorem A.43 of \citep{bai2010spectral}] \label{lemma:rank_ineq}
    Let $\mathbf{Q}$ and $\mathbf{R}$ be two $p \times p$ Hermitian matrices. Then,
    \begin{equation}
        \|F^{\mathbf{Q}} - F^{\mathbf{R}}\|_{\infty} \leq \frac{1}{p} \mathrm{rank}(\mathbf{Q} -\mathbf{R})
    \end{equation}
\end{lemma}

\section{Proof of Step 1 \eqref{step1}}\label{pro3}
Let us adapt the proof presented in Section 3.3.2 of \citep{bai2010spectral}. We denote $\E_k$ as the conditional expectation given $\{\mathbf{x}_{k+1}, \dots, \mathbf{x}_n\}$. We can then write:

\begin{equation}
    \g^{p}(z) - \E(\g^p(z)) = \frac{1}{p} \sum_{k=1}^n \gamma_k,
\end{equation}
where

\begin{equation}
    \gamma_k = (\E_k - \E_{k-1})\big[\tr(\mathbf{H} - z)^{-1} - \tr(\mathbf{H}_{k} - z)^{-1}\big],
\end{equation}
and

\begin{equation}
    \mathbf{H}_k = \mathbf{H} - \sum_{i<k} \boldsymbol{\bar{\Theta}}_{(ik)} \boldsymbol{\bar{\Theta}}_{(ik)}^\top - \sum_{k<j} \boldsymbol{\bar{\Theta}}_{(kj)} \boldsymbol{\bar{\Theta}}_{(kj)}^\top.
\end{equation}
By using Theorem A.5 of \citep{bai2010spectral} several times, one can find a constant $C > 0$ independent of $k$ such that:

\begin{equation}
    |\gamma_k| \leq \frac{Cn}{|\Im{z}|}.
\end{equation}
Let us fix $m > 1$. Lemma 2.12 of \citep{bai2010spectral} provides a constant $K_m$ such that:

\begin{equation}
    \E\left[\left|\sum_{k=1}^n \gamma_k\right|^m\right] \leq K_m \E\left[\left(\sum_{k=1}^n |\gamma_k|^2\right)^{m/2}\right].
\end{equation}
Hence,

\begin{equation}
    \E\left[|\g^{p}(z) - \E(\g^p(z))|^m\right] \leq \frac{K_m}{p^m} \left(n\frac{C^2n^2}{|\Im{z}|^2}\right)^{m/2}.
\end{equation}
Since $n \underset{p\to +\infty}{=}  O(\sqrt{p})$, we have:

\begin{equation}
    \E\left[|\g^{p}(z) - \E(\g^p(z))|^m\right] = O\left(\frac{1}{p^{m/4}}\right).
\end{equation}
By taking $m > 4$, the Borel-Cantelli lemma yields:

\begin{equation}
    \g^{p}(z) - \E(\g^p(z)) \underset{p\to +\infty}{\longrightarrow} 0
\end{equation}
almost surely.

\bibliographystyle{apsrev4-2} 
\bibliography{sample}

\begin{thebibliography}{10}%
\makeatletter
\providecommand \@ifxundefined [1]{%
 \@ifx{#1\undefined}
}%
\providecommand \@ifnum [1]{%
 \ifnum #1\expandafter \@firstoftwo
 \else \expandafter \@secondoftwo
 \fi
}%
\providecommand \@ifx [1]{%
 \ifx #1\expandafter \@firstoftwo
 \else \expandafter \@secondoftwo
 \fi
}%
\providecommand \natexlab [1]{#1}%
\providecommand \enquote  [1]{``#1''}%
\providecommand \bibnamefont  [1]{#1}%
\providecommand \bibfnamefont [1]{#1}%
\providecommand \citenamefont [1]{#1}%
\providecommand \href@noop [0]{\@secondoftwo}%
\providecommand \href [0]{\begingroup \@sanitize@url \@href}%
\providecommand \@href[1]{\@@startlink{#1}\@@href}%
\providecommand \@@href[1]{\endgroup#1\@@endlink}%
\providecommand \@sanitize@url [0]{\catcode `\\12\catcode `\$12\catcode
  `\&12\catcode `\#12\catcode `\^12\catcode `\_12\catcode `\%12\relax}%
\providecommand \@@startlink[1]{}%
\providecommand \@@endlink[0]{}%
\providecommand \url  [0]{\begingroup\@sanitize@url \@url }%
\providecommand \@url [1]{\endgroup\@href {#1}{\urlprefix }}%
\providecommand \urlprefix  [0]{URL }%
\providecommand \Eprint [0]{\href }%
\providecommand \doibase [0]{https://doi.org/}%
\providecommand \selectlanguage [0]{\@gobble}%
\providecommand \bibinfo  [0]{\@secondoftwo}%
\providecommand \bibfield  [0]{\@secondoftwo}%
\providecommand \translation [1]{[#1]}%
\providecommand \BibitemOpen [0]{}%
\providecommand \bibitemStop [0]{}%
\providecommand \bibitemNoStop [0]{.\EOS\space}%
\providecommand \EOS [0]{\spacefactor3000\relax}%
\providecommand \BibitemShut  [1]{\csname bibitem#1\endcsname}%
\let\auto@bib@innerbib\@empty
\bibitem [{\citenamefont {Marčenko}\ and\ \citenamefont
  {Pastur}(1967)}]{marchenko1967distribution}%
  \BibitemOpen
  \bibfield  {author} {\bibinfo {author} {\bibfnamefont {V.~A.}\ \bibnamefont
  {Marčenko}}\ and\ \bibinfo {author} {\bibfnamefont {L.~A.}\ \bibnamefont
  {Pastur}},\ }\href {https://doi.org/10.1070/SM1967v001n04ABEH001994}
  {\bibfield  {journal} {\bibinfo  {journal} {Mathematics of the USSR-Sbornik}\
  }\textbf {\bibinfo {volume} {1}},\ \bibinfo {pages} {457} (\bibinfo {year}
  {1967})}\BibitemShut {NoStop}%
\bibitem [{\citenamefont {Bandeira}\ \emph {et~al.}(2017)\citenamefont
  {Bandeira}, \citenamefont {Lodhia},\ and\ \citenamefont
  {Rigollet}}]{bandeira2017marvcenkopastur}%
  \BibitemOpen
  \bibfield  {author} {\bibinfo {author} {\bibfnamefont {A.~S.}\ \bibnamefont
  {Bandeira}}, \bibinfo {author} {\bibfnamefont {A.}~\bibnamefont {Lodhia}},\
  and\ \bibinfo {author} {\bibfnamefont {P.}~\bibnamefont {Rigollet}},\ }\href
  {https://doi.org/10.1214/17-ECP59} {\bibfield  {journal} {\bibinfo  {journal}
  {Electronic Communications in Probability}\ }\textbf {\bibinfo {volume}
  {22}},\ \bibinfo {pages} {1 } (\bibinfo {year} {2017})}\BibitemShut {NoStop}%
\bibitem [{\citenamefont {Bao}(2019)}]{tracy_widom_kendall}%
  \BibitemOpen
  \bibfield  {author} {\bibinfo {author} {\bibfnamefont {Z.}~\bibnamefont
  {Bao}},\ }\href {https://doi.org/10.1214/18-AOS1786} {\bibfield  {journal}
  {\bibinfo  {journal} {The Annals of Statistics}\ }\textbf {\bibinfo {volume}
  {47}},\ \bibinfo {pages} {3504 } (\bibinfo {year} {2019})}\BibitemShut
  {NoStop}%
\bibitem [{\citenamefont {Li}\ \emph {et~al.}(2021)\citenamefont {Li},
  \citenamefont {Wang},\ and\ \citenamefont {Li}}]{clt_lss}%
  \BibitemOpen
  \bibfield  {author} {\bibinfo {author} {\bibfnamefont {Z.}~\bibnamefont
  {Li}}, \bibinfo {author} {\bibfnamefont {Q.}~\bibnamefont {Wang}},\ and\
  \bibinfo {author} {\bibfnamefont {R.}~\bibnamefont {Li}},\ }\href
  {https://doi.org/10.1214/20-AOS2013} {\bibfield  {journal} {\bibinfo
  {journal} {The Annals of Statistics}\ }\textbf {\bibinfo {volume} {49}},\
  \bibinfo {pages} {1569 } (\bibinfo {year} {2021})}\BibitemShut {NoStop}%
\bibitem [{\citenamefont {Li}\ \emph {et~al.}(2023)\citenamefont {Li},
  \citenamefont {Wang},\ and\ \citenamefont {Wang}}]{li2023eigenvalues}%
  \BibitemOpen
  \bibfield  {author} {\bibinfo {author} {\bibfnamefont {Z.}~\bibnamefont
  {Li}}, \bibinfo {author} {\bibfnamefont {C.}~\bibnamefont {Wang}},\ and\
  \bibinfo {author} {\bibfnamefont {Q.}~\bibnamefont {Wang}},\ }\href@noop {}
  {\bibfield  {journal} {\bibinfo  {journal} {Science China Mathematics}\
  }\textbf {\bibinfo {volume} {66}},\ \bibinfo {pages} {2615} (\bibinfo {year}
  {2023})}\BibitemShut {NoStop}%
\bibitem [{\citenamefont {Espana}\ \emph {et~al.}(2024)\citenamefont {Espana},
  \citenamefont {Coz},\ and\ \citenamefont {Smerlak}}]{kendall_markowitz}%
  \BibitemOpen
  \bibfield  {author} {\bibinfo {author} {\bibfnamefont {T.}~\bibnamefont
  {Espana}}, \bibinfo {author} {\bibfnamefont {V.~L.}\ \bibnamefont {Coz}},\
  and\ \bibinfo {author} {\bibfnamefont {M.}~\bibnamefont {Smerlak}},\
  }\href@noop {} {\bibfield  {journal} {\bibinfo  {journal} {arXiv preprint
  arXiv:2410.17366}\ } (\bibinfo {year} {2024})}\BibitemShut {NoStop}%
\bibitem [{\citenamefont {Bai}\ and\ \citenamefont
  {Silverstein}(2010)}]{bai2010spectral}%
  \BibitemOpen
  \bibfield  {author} {\bibinfo {author} {\bibfnamefont {Z.}~\bibnamefont
  {Bai}}\ and\ \bibinfo {author} {\bibfnamefont {J.~W.}\ \bibnamefont
  {Silverstein}},\ }\href@noop {} {\emph {\bibinfo {title} {Spectral analysis
  of large dimensional random matrices}}},\ Vol.~\bibinfo {volume} {20}\
  (\bibinfo  {publisher} {Springer},\ \bibinfo {year} {2010})\BibitemShut
  {NoStop}%
\bibitem [{\citenamefont {Hoeffding}(1948)}]{hoeffding}%
  \BibitemOpen
  \bibfield  {author} {\bibinfo {author} {\bibfnamefont {W.}~\bibnamefont
  {Hoeffding}},\ }\href {https://doi.org/10.1214/aoms/1177730196} {\bibfield
  {journal} {\bibinfo  {journal} {The Annals of Mathematical Statistics}\
  }\textbf {\bibinfo {volume} {19}},\ \bibinfo {pages} {293 } (\bibinfo {year}
  {1948})}\BibitemShut {NoStop}%
\bibitem [{\citenamefont {Daniels}(1944)}]{daniels}%
  \BibitemOpen
  \bibfield  {author} {\bibinfo {author} {\bibfnamefont {H.~E.}\ \bibnamefont
  {Daniels}},\ }\href {http://www.jstor.org/stable/2334112} {\bibfield
  {journal} {\bibinfo  {journal} {Biometrika}\ }\textbf {\bibinfo {volume}
  {33}},\ \bibinfo {pages} {129} (\bibinfo {year} {1944})}\BibitemShut
  {NoStop}%
\bibitem [{\citenamefont {Benaych-Georges}\ and\ \citenamefont
  {Espana}(2025)}]{benaych2025spectral}%
  \BibitemOpen
  \bibfield  {author} {\bibinfo {author} {\bibfnamefont {F.}~\bibnamefont
  {Benaych-Georges}}\ and\ \bibinfo {author} {\bibfnamefont {T.}~\bibnamefont
  {Espana}},\ }\href@noop {} {\bibfield  {journal} {\bibinfo  {journal} {arXiv
  preprint arXiv:2509.25551}\ } (\bibinfo {year} {2025})}\BibitemShut {NoStop}%
\end{thebibliography}%

\end{document}